\documentclass{article}
\usepackage{amsmath,amssymb,amsthm}
\usepackage{hyperref}

\newtheorem{theorem}{Theorem}

\newtheorem{lemma}{Lemma}
\newtheorem{proposition}{Proposition}

\newtheorem*{observation}{Observation}

\theoremstyle{definition}

\begin{document}

\title{Partition and sum is fast}
\date{\empty}
%\author{}
\author{Steve Butler \and Ron Graham \and Richard Stong}
\author{
Steve Butler\thanks{Department of Mathematics, Iowa State University, Ames, IA 50011, USA
(butler@iastate.edu).}
\and
Ron Graham\thanks{Department of Computer Science and Engineering, UC San Diego, La Jolla, CA 92093, USA (graham@ucsd.edu).}
\and
Richard Stong\thanks{Center for Communications Research, La Jolla, CA 92121, USA (stong@ccrwest.org)}}

\maketitle

\begin{abstract}
We consider the following ``partition and sum'' operation on a natural number:  Treating the number as a long string of digits insert several plus signs in between some of the digits and carry out the indicated sum.  This results in a smaller number and repeated application can always reduce the number to a single digit.  We show that surprisingly few iterations of this operation are needed to get down to a single digit.
\end{abstract}

\section{Introduction}
Consider the following operation, we call \emph{partition and sum}, that can be performed on a natural number:  

\begin{quote}
\emph{Treating the number as a long string of digits insert plus signs, ``${+}$'', in between some of the digits (as many or as few as desired) and then carry out the indicated sum to produce a new number.}
\end{quote}

This operation can be done for a number in any base (where we use the same base for the entire process), if there is a possibility of confusion we will indicate which base a number is written in using subscript notation, i.e., $111_{(2)}$ is $7$ in base $2$.  The operation in base $2$ was originally suggested by Gregory Galperin (see \cite{MSRI}) who asked for a bound on how many steps it takes to get a number down to $1$ (as long as one plus sign is inserted the number will decrease and so we always can get to $1$).

We encourage the reader at this point to put the paper aside, write down some random binary numbers, and try to use the partition and sum operation to get to $1$ in a few steps.

One approach that seems to work well (and that you might have tried) is to simply insert all possible plus signs, i.e., sum the digits.  This takes a number $n$ and gets it to something on the order of at most $\log n$.  So then we only need to apply enough iterations so that when we iteratively apply the logarithm to $n$, i.e., $\log(\log(\cdots(\log(n))\cdots))$ the value is less than $1$.  This is known as $\log^*(n)$ and grows \emph{amazingly} slowly with $n$, i.e., goes to infinity slower than just about any function we would expect to encounter.  But it still goes to infinity.

While the simple sum the digits strategy gives a good bound, it is still \emph{far} from the truth!

% As long as we insert at least one plus sign, then the new number will be smaller.  So with enough applications we can always reduce down to a single digit.  One easy way to do this is to insert plus signs between each pair of digits, i.e., sum the digits. For example if we start with $31415926535897932384626433832795$ then adding up all the digits gives $155$, which adding those digits gives $11$, which adding those digits gives us $2$.  So we only needed three operations to get to a single digit, but of course we can do better, for instance if in the first step we replace `${+}5{+}9{+}$'' with ``${+}59{+}$'' then we only need two operations.  Of course we can be more creative:
% \begin{align*}
% 3141{+}592{+}65358{+}9793{+}23846{+}2643{+}3832{+}795&=110000\\
% 1{+}1{+}0000&=2
% \end{align*}
% We again ended up with a final answer of $2$. In fact no matter how the operations are performed, the final digit will always be the same because the value is invariant under modulo $b-1$.  (This is the same principle which states a number is divisible by $9$ if and only if the sum of digits is divisible by $9$.)

% For small values it is not surprising to only need a few steps, but one might expect that as the numbers grow very large that we would need to take more and more steps to get down to a single digit.  This turns out to not be true!

\begin{theorem}\label{thm:base}
In base $b=2$ we can take any natural number to a single digit in at most \emph{two} applications of partition and sum.  In base $b\ge3$ we can take any natural number to a single digit in at most \emph{three} applications of partition and sum.
\end{theorem}

%The problem for base $2$ comes from Gregory Galperin and was featured in the Puzzles Column of \emph{Emissary} \cite{MSRI}.   In what follows we show how to establish Theorem~\ref{thm:base}.

\section{Reducing in base $2$}
For the result in base $2$ it suffices to show that in one application of partition and sum we can get to a power of two (i.e., $100\ldots 0_{(2)}$). We do this by changing the way we think about the operation.  Namely we start by inserting all possible plus signs between digits and then removing some of them to \emph{merge} digits together forming a larger number and increasing the sum (so now we are doing ``merging and sum'').  We will make use of the following observation (here ``${*}$'' indicates an unknown digit).

\begin{observation}
Merging $1{+}{*}$ to $1{*}_{(2)}$ increases the sum by $1$; merging $1{+}0{+}{*}$ to $10{*}_{(2)}$ increases the sum by $3$; merging $1{+}1{+}{*}$ to $11{*}_{(2)}$ increases the sum by $4$. 
\end{observation}

One important thing we will make use of when we  merge \emph {three} terms together is that we get significantly more value for our $1$'s.  In particular the total sum from the parts which were merged together in triples is at least $7/3$  of the number of individual $1$'s that were used (i.e., $1{+}1{+}{1}$ to $111_{(2)}$ changed that portion of the sum from $3$ to $7$ and the rest of the possibilities yield even better returns).  So by forming many triples we can efficiently increase the sum; this forms the basis of our strategy.

\begin{quote}
\textbf{Triple merging strategy:}  Given a number whose binary expansion has $m$ total $1$'s where $2^k<m\le 2^{k+1}$, insert plus signs between each pair of digits in the expansion and perform the following as long as the sum is $\le 2^{k+1}$:  Find the left-most $1$ that has not yet been merged and merge it with its two successors.

Finally, use merges of the form $1{*}$ with the remaining digits to make up the difference to get the sum to $2^{k+1}$.
\end{quote}

We note this strategy will \emph{fail} when $m=5$.  In this case if there is a $1{+}0{+}{*}$ anywhere then by merging the three terms together we get the sum to $8$ and we are done.  This leaves $111110_{(2)}$ and $11111_{(2)}$ which can be handled by $11_{(2)}{+}11_{(2)}{+}10_{(2)}$ and $1_{(2)}+1111_{(2)}$ respectively.

\begin{proposition}
The triple merging strategy works when $m\ne 5$.
\end{proposition}
\begin{proof}
If $\frac73(m-5) + 5> 2^{k+1}$ then this strategy must succeed.  To see this, suppose that $q$ is the number of $1$'s that have been merged into triples at some stage in the strategy.  Then the total sum is at least $\frac73q+(m-q)$.  We will stop just before we get above $2^{k+1}$ and the inequality indicates that when we stop we have at least \emph{six} $1$'s remaining (for otherwise our sum would be too large).  The only reason we would stop though is if we ran out of $1$'s (which we haven't), or the next merging of triples would put us over $2^{k+1}$.  Since merging triples increases the total by at most $4$ this means that we are now at most $3$ away.  So using the remaining $1$'s we can now form $1{*}$'s to make up the difference and get the sum to $2^{k+1}$.

Simplifying we can conclude that the strategy works for $m>\frac672^k+\frac{20}7$.  When we combine this with $m\ge2^k+1$, we can conclude that this holds for $m\ge10$.  On the other hand for $m=1,2,3,4,6,7,8$ we could never form a triple and so the forming twins portion of the strategy kicks in and this will always succeed.

That leaves us with $m=9$, and to finish this off we consider the possible triples that the strategy gives us before we have to move to forming twins.

\bigskip
\noindent\hfil
\begin{tabular}{|c|c|c|}\hline 
Starting triples&Remaining $1$'s&Remaining difference\\ \hline\hline
$11{*}$&$\ge6$&$3$\\\hline
$11{*}$~~~$10{*}$&$\ge 4$&$0$\\\hline
$10{*}$~~~$11{*}$&$\ge 4$&$0$\\\hline
$10{*}$~~~$10{*}$&$\ge 5$&$1$\\\hline
\end{tabular}

\bigskip

In each case the number of $1$'s that remain can readily be used to make up the difference.  This finishes the case for $m=9$ and also the proof.
\end{proof}

\section{Reducing in base $\ge 4$}
We start by showing that if $n$ is small then there is a simple strategy that works for partition and sum.

\begin{lemma}\label{lem:4a}
Let $b\ge 4$ be our base.  Then any number $n<3b^2-b-1$ can be collapsed to a single digit in at most two steps.
\end{lemma}
\begin{proof}
First we observe that for $n=1,2,\ldots,2b-2$,  we can  apply the summing digits strategy to get to a single digit in one step. The first number for which the summing digits strategy fails to reach
a single digit in two steps is $1(b-1)(b-1)_{(b)}$. However, for this number we can first do $1_{(b)}+(b-1)(b-1)_{(b)}=100_{(b)}$ and then sum our digits as before.

The number $2(b-2)(b-1)_{(b)}=3b^2-b-1$ is the next time where the summing digits strategy fails, establishing the lemma.
\end{proof}

This lemma is tight; the number $2(b-2)(b-1)_{(b)}$ takes three steps as is easy to check.  In fact more is true.

\begin{observation}
Let $b\ge 4$ be our base.  Then $20\ldots0(b-2)(b-1)_{(b)}$ takes three steps for \emph{any} number of zeroes.  In particular, there are infinitely many numbers that take three steps.
\end{observation}

To see this we first note that the process is always invariant modulo $(b-1)$ (this is the same principle which states that a number is divisible by $9$ if and only if the sum of the digits is divisible by $9$).  So when we have taken this particular number to a single digit then we will end with $1$.  Now if it could be done in two steps we would have to be able to get it to a number of the form $10\ldots0_{(b)}$, i.e., a power of $b$.  In particular the last digit after the first step would have to be zero. However, for a number of this form the last digit in merging would come from $b-1$, $(b-1)+(b-2)$, $(b-1)+2$ or $(b-1)+(b-2)+2$ and none of these are $0$ modulo $b$ when $b\ge 4$.  (For $b=3$ we can get a $0$ in the last digit and so this number can be handled in two steps.)

Now we see that if the sum of digits is small then we can apply Lemma~\ref{lem:4a} after doing one step of summing the digits.  We next show that if the sum of digits is large then we can take one  step to get us to a number whose form can be finished in at most two more steps.

\begin{lemma}\label{lem:4b}
Let $b\ge 4$ be our base and let $n$ a number with the sum of its digits $m\ge b^2$.  Then in one step, $n$ can be collapsed to a number of the form $c0\ldots0de$ where $c\le 2$ and $de_{(b)}\le b^2-2b$.
\end{lemma}
\begin{proof}
Let $n=(\ldots a_4a_3a_2a_1a_0)_{(b)}$ and let $$A=\max\{a_1+a_3+a_5+\cdots,a_2+a_4+a_6+\cdots\}.$$  We do not use the last digit for $A$ and so we have $A\ge (m-(b-1))/2$.  We now consider what happens if we merge in pairs  so that the leading digits in the pairs sum to $A$ (i.e., we pair so that all the digits are even or odd depending on which gave us $A$).  The sum total of this merging strategy will be
\[
(b-1)A+m\ge {(b-1)(m-b+1)\over2}+m ={mb+m-(b-1)^2\over2}> {mb\over2}.
\]
(The last step is by our assumption that $m\ge b^2$.)  If we now break these pairs one at a time, say from left to right, then the difference in the total would be at most $(b-1)^2$ at each pair.  Therefore we have a sequence of ways to partition which go from $m$ to $(b-1)A+m$ where the difference between two consecutive methods is at most $(b-1)^2$.

Now $m$ is in an interval of the form $[b^t,2b^t)$ or $[2b^t,b^{t+1})$ for some $t\ge 2$.  However we have that $(b-1)A+m>bm/2$ cannot be in the same interval (here we use that $b\ge 4$).  Therefore there will be some smallest merging strategy which will exceed the top of the given range containing $m$.  Let $M$ be the resulting total using this merging.  Then we either have $2b^t\le M<2b^t+(b-1)^2$ or $b^{t+1}\le M< b^{t+1}+(b-1)^2$, depending on which case we are in, which gives exactly the sort of base $b$ representation as given in the statement of the lemma.
\end{proof}

Now if the sum of digits $m<b^2$ then sum the singletons and  apply Lemma~\ref{lem:4a} to the result and we take at most three steps.  On the other hand, if the sum of digits $m\ge b^2$ then by Lemma~\ref{lem:4b} in one step  we will collapse to a number of the form $c0\ldots 0de$ with $c\le 2$ and $db+e\le b(b-2)$.  We have that $d+e\le (b-3)+(b-1)=2b-4$, and so adding all the singletons gives a number of size at most $2b-2$, which can be finished in one more step.

\section{Reducing in base $3$}
We can adopt the same ideas as we have seen before by first showing when the number is small then at most two steps of partition and sum suffice.  Then either the sum of the digits is small and so we first sum the digits and then apply a known two step technique \emph{or} the sum of digits is large meaning we have many digits to work with and so we have a lot of flexibility in merging that allows us in one step to get to a number which can readily be done in at most two steps using the sum of digits strategy.

The details of this are not enlightening and so we omit them here and refer interested readers to \cite{bgs} for details.  But what makes base $3$ interesting is that while there are numbers that require three steps, there are only eleven of them!

\begin{theorem}[{Butler-Graham-Stong \cite{bgs}}]
In base $3$ any natural number can be collapsed to a single digit in at most two applications of partition and sum \emph{except} for the following eleven:
\[
\begin{array}{r@{\,=\,}l@{\qquad\qquad\quad}r@{\,=\,}l}
1781 & 2102222_{(3)} & 41065 & 2002022221_{(3)}\\[3pt]
3239 & 11102222_{(3)} & 43981 & 2020022221_{(3)}\\[3pt]
3887 & 12022222_{(3)} & 98657 & 12000022222_{(3)}\\[3pt]
11177 & 120022222_{(3)} & 131461 & 20200022221_{(3)}\\[3pt]
14821 & 202022221_{(3)} & 393901 & 202000022221_{(3)}\\[3pt]
33047 & 1200022222_{(3)}
\end{array}
\]
\end{theorem}

This was proved using a computer to exhaustively handle the case when the sum of digits is small and using theory to handle the case when the sum of digits is large.  Finding a simpler, non-computer proof, for the base $3$ case would be an interesting problem.

\end{document}